\documentclass[12pt,a4paper]{amsart}
\usepackage{amsmath,amstext,amssymb,amsthm,amsfonts}
\newcommand{\nc}{\newcommand}

\nc{\one}{\mbox{\bf 1}}
\nc{\invtensor}{\underset{\leftarrow}{\otimes}}
\nc{\const}{\operatorname{const}}

\nc{\ad}{\operatorname{ad}}
\nc{\tr}{\operatorname{tr}}
\nc{\tp}{\operatorname{top}}
\nc{\rank}{\operatorname{rank}}
\nc{\corank}{\operatorname{corank}}
\nc{\codim}{\operatorname{codim}}
\nc{\sdim}{\operatorname{sdim}}
\nc{\mult}{\operatorname{mult}}

\nc{\spn}{\operatorname{span}}
\nc{\Sym}{\operatorname{Sym}}
\nc{\sym}{\operatorname{sym}}
\nc{\id}{\operatorname{id}}
\nc{\Id}{\operatorname{Id}}
\nc{\Ree}{\operatorname{Re}}

\nc{\htt}{\operatorname{ht}}
\nc{\Ker}{\operatorname{Ker}}
\nc{\rker}{\operatorname{rKer}}
\nc{\im}{\operatorname{Im}}
\nc{\osp}{\mathfrak{osp}}
\nc{\sgn}{\operatorname{sgn}}
\nc{\F}{\operatorname{F}}
\nc{\Mod}{\operatorname{Mod}}
\nc{\Mat}{\operatorname{Mat}}
\nc{\Soc}{\operatorname{Soc}}
\nc{\Inj}{\operatorname{Inj}}
\nc{\Hom}{\operatorname{Hom}}
\nc{\End}{\operatorname{End}}
\nc{\supp}{\operatorname{supp}}
\nc{\Card}{\operatorname{Card}}
\nc{\Ann}{\operatorname{Ann}}
\nc{\Ind}{\operatorname{Ind}}
\nc{\Coind}{\operatorname{Coind}}
\nc{\wt}{\operatorname{wt}}
\nc{\ch}{\operatorname{ch}}
\nc{\Stab}{\operatorname{Stab}}
\nc{\Sch}{{\mathcal S}\mbox{\em ch}}
\nc{\Irr}{\operatorname{Irr}}
\nc{\Spec}{\operatorname{Spec}}
\nc{\Prim}{\operatorname{Prim}}
\nc{\Aut}{\operatorname{Aut}}
\nc{\Ext}{\operatorname{Ext}}

\nc{\Fract}{\operatorname{Fract}}
\nc{\gr}{\operatorname{gr}}
\nc{\deff}{\operatorname{def}}
\nc{\HC}{\operatorname{HC}}

\nc{\red}{\operatorname{red}}

\nc{\wdchi}{\widetilde{\chi}}
\nc{\wdH}{\widetilde{H}}
\nc{\wdN}{\widetilde{N}}
\nc{\wdM}{\widetilde{M}}
\nc{\wdO}{\widetilde{O}}
\nc{\wdR}{\widetilde{R}}
\nc{\wdS}{\widetilde{S}}
\nc{\wdV}{\widetilde{V}}

\nc{\wdC}{\widetilde{C}}

\nc{\Obj}{\operatorname{Obj}}
\nc{\Dglie}{\operatorname{{\mathcal D}glie}}
\nc{\Fin}{\operatorname{{\mathcal F}in}}

\nc{\Adm}{\operatorname{\mathcal{A}dm}}

\nc{\Sg}{{\cS(\fg)}}
\nc{\Shg}{{\cS(\fhg)}}
\nc{\Ug}{{\cU(\fg)}}
\nc{\Uhg}{{\cU(\fhg)}}
\nc{\Sh}{{\cS(\fh)}}
\nc{\Uh}{{\cU(\fh)}}
\nc{\Uhh}{{\cU(\fhh)}}
\nc{\Zg}{{{\mathcal{Z}}(\fg)}}

\nc{\Vir}{{\mathcal{V}ir}}
\nc{\NS}{{\mathcal{N}S}}

\nc{\tZg}{{\widetilde{\mathcal Z}({\mathfrak g})}}
\nc{\Zk}{{\mathcal Z}({\mathfrak k})}

\nc{\Up}{{\mathcal U}({\mathfrak p})}
\nc{\Ah}{{\mathcal A}({\mathfrak h})}
\nc{\Ag}{{\mathcal A}({\mathfrak g})}
\nc{\Ap}{{\mathcal A}({\mathfrak p})}
\nc{\Zp}{{\mathcal Z}({\mathfrak p})}
\nc{\cR}{\mathcal R}
\nc{\cS}{\mathcal S}
\nc{\cT}{\mathcal{T}}
\nc{\cY}{\mathcal Y}
\nc{\cA}{\mathcal A}
\nc{\cU}{\mathcal U}
\nc{\cH}{\mathcal H}
\nc{\cM}{\mathcal M}
\nc{\cL}{\mathcal L}
\nc{\cF}{\mathcal F}
\nc{\fg}{\mathfrak g}

\nc{\fo}{\mathfrak o}

\nc{\CO}{\mathcal O}
\nc{\CR}{\mathcal R}
\nc{\Cl}{\mathcal {C}\ell}

\nc{\cW}{\mathcal{W}}
\nc{\bM}{\mathbf{M}}
\nc{\bL}{\mathbf{L}}
\nc{\bN}{\mathbf{N}}

\nc{\zq}{\mathpzc q}

\nc{\fl}{\mathfrak l}
\nc{\fn}{\mathfrak n}
\nc{\fm}{\mathfrak m}
\nc{\fp}{\mathfrak p}
\nc{\fh}{\mathfrak h}
\nc{\ft}{\mathfrak t}
\nc{\fk}{\mathfrak k}
\nc{\fb}{\mathfrak b}
\nc{\fs}{\mathfrak s}
\nc{\fB}{\mathfrak B}

\nc{\vareps}{\varepsilon}
\nc{\varesp}{\varepsilon}
\nc{\veps}{\varepsilon}

\nc{\fsl}{\mathfrak{sl}}
\nc{\fgl}{\mathfrak{gl}}
\nc{\fso}{\mathfrak{so}}

\nc{\fpq}{\mathfrak{pq}}
\nc{\fq}{\mathfrak q}
\nc{\fsq}{\mathfrak{sq}}
\nc{\fpsq}{\mathfrak{psq}}

\nc{\fhg}{\hat{\fg}}
\nc{\fhn}{\hat{\fn}}
\nc{\fhh}{\hat{\fh}}
\nc{\fhb}{\hat{\fb}}
\nc{\hrho}{\hat{\rho}}

\nc{\hsl}{\hat{\fsl}}
\nc{\fpo}{\mathfrak{po}}
\nc{\dirlim}{\underset{\rightarrow}{\lim}\,}
\nc{\nen}{\newenvironment}
\nc{\ol}{\overline}
\nc{\ul}{\underline}
\nc{\ra}{\rightarrow}
\nc{\lra}{\longrightarrow}
\nc{\Lra}{\Longrightarrow}

\nc{\Lla}{\Longleftarrow}

\nc{\Llra}{\Longleftrightarrow}

\nc{\thla}{\twoheadleftarrow}

\nc{\hra}{\hookrightarrow}

\nc{\iso}{\overset{\sim}{\lra}}

\nc{\ssubset}{\underset{\not=}{\subset}}

\nc{\vac}{|0\rangle}

\nc{\Thm}[1]{Theorem~\ref{#1}}
\nc{\Prop}[1]{Proposition~\ref{#1}}
\nc{\Lem}[1]{Lemma~\ref{#1}}
\nc{\Cor}[1]{Corollary~\ref{#1}}
\nc{\Conj}[1]{Conjecture~\ref{#1}}
\nc{\Claim}[1]{Claim~\ref{#1}}
\nc{\Defn}[1]{Definition~\ref{#1}}
\nc{\Exa}[1]{Example~\ref{#1}}
\nc{\Rem}[1]{Remark~\ref{#1}}
\nc{\Note}[1]{Note~\ref{#1}}
\nc{\Quest}[1]{Question~\ref{#1}}
\nc{\Hyp}[1]{Hypoth\`ese~\ref{#1}}
\nen{thm}[1]{\label{#1}{\bf Theorem.\ } \em}{}
\nen{prop}[1]{\label{#1}{\bf Proposition.\ } \em}{}
\nen{lem}[1]{\label{#1}{\bf Lemma.\ } \em}{}
\nen{cor}[1]{\label{#1}{\bf Corollary.\ } \em}{}
\nen{conj}[1]{\label{#1}{\bf Conjecture.\ } \em}{}

\nen{claim}[1]{\label{#1}{\bf Claim.\ } \em}{}

\nen{defn}[1]{\label{#1}{\bf Definition.\ } }{}
\nen{exa}[1]{\label{#1}{\bf Example.\ } }{}
\nen{rem}[1]{\label{#1}{\em Remark.\ } }{}
\nen{note}[1]{\label{#1}{\em Note.\ } }{}
\nen{exer}[1]{\label{#1}{\em Exercise.\ } }{}
\nen{sket}[1]{\label{#1}{\em Sketch of proof.\ } }{}
\nen{quest}[1]{\label{#1}{\bf Question.\ } \em}{}

\nen{hyp}[1]{\label{#1}{\bf Hypoth\`ese.\ } \em}{}
\begin{document}

\setcounter{section}{-1}

\title[Affine  denominator identity]{Weyl denominator identity
for affine Lie superalgebras with non-zero dual Coxeter number}
\author[Maria Gorelik]{Maria Gorelik}

\address{Dept. of Mathematics, The Weizmann Institute of Science,
Rehovot 76100, Israel}
\email{maria.gorelik@weizmann.ac.il}
\thanks{Supported in part by ISF Grant No. 1142/07}

\begin{abstract}
Weyl denominator identity for the affinization of a basic 
Lie superalgebra with non-zero Killing form
was formulated by V.~Kac and M.~Wakimoto and was proven
by them for the defect one case. In this paper we prove this identity.
\end{abstract}

\maketitle
\section{Introduction}\label{intro}
Let $\fg$ be a basic Lie superalgebra with a non-zero Killing form. 
Let $\fhg$ be the affinization of $\fg$. Let $\fh$ (resp., $\fhh$) be the
Cartan subalgebra in $\fg$ (resp., in $\fhg$) and
let $(-,-)$ be the bilinear form
on $\fhh^*$ which is induced by the Killing form on $\fg$.
Let $\Delta$ (resp., $\hat{\Delta}$) be the root system
of $\fg$ (resp., of $\fhg$). We set 
$$\Delta^{\#}:=\{\alpha\in \Delta_{\ol{0}}|\ (\alpha,\alpha)>0\}.$$
Then $\Delta^{\#}$ is a root system of a simple Lie algebra. 
Let $\hat{\Delta}^{\#}$ be the affinization
of $\Delta^{\#}$.
Denote by $\hat{W}^{\#}$ (resp., ${W}^{\#}$)
the subgroup of $GL(\fhh)$
generated by the reflections $s_{\alpha}: \alpha\in\hat{\Delta}_{\ol{0}},\ 
(\alpha,\alpha)>0$ (resp., $s_{\alpha}: \alpha\in {\Delta}^{\#}$). 
Then ${W}^{\#}$ is the Weyl group of $\Delta^{\#}$ and
$\hat{W}^{\#}$ is the corresponding affine Weyl group.
Recall that $\hat{W}^{\#}=W^{\#}\ltimes T$, where 
$T\subset \hat{W}^{\#}$ is the translation group, 
see~\cite{Kbook}, Chapter 6.
Let $\Pi$ be a set of simple roots for $\fg$, and
let $\hat{\Pi}=\Pi\cup\{\alpha_0\}$ be the corresponding set of simple roots
for $\fhg$.  Let $\Delta_+$,
$\hat{\Delta}_+$ be the corresponding sets of positive roots. We set
$$R:=\frac{\prod_{\alpha\in\Delta_{+,0}}(1-e^{-\alpha})}
{\prod_{\alpha\in\Delta_{+,1}}(1+e^{-\alpha})},\ \ \ 
\hat{R}:=\frac{\prod_{\alpha\in\hat{\Delta}_{+,0}}(1-e^{-\alpha})}
{\prod_{\alpha\in\hat{\Delta}_{+,1}}(1+e^{-\alpha})}.$$
Following~\cite{KW}, we call $R$ the {\em Weyl denominator}
and $\hat{R}$ the {\em affine Weyl denominator}.
The Weyl denominator identity conjectured by V.~Kac and M.~Wakimoto
in~\cite{KW} can be written as
$$\hat{R}e^{\hat{\rho}}=\sum_{w\in T} w(Re^{\hat{\rho}}),$$
where $\hat{\rho}\in\fhh^*$ is such that 
$2(\hat{\rho},\alpha)=(\alpha,\alpha)$ for each $\alpha\in\Pi$.
The original form of this identity is given in formula~(\ref{denom}).
In this paper we prove this identity.

In the paper~\cite{G} we proved the analog of Weyl denominator
identity for finite-dimensional Lie superalgebras (also formulated
and partially proven by Kac-Wakimoto). The proof of the present result 
makes use of this version of Weyl denominator identity.

{\em Acknowledgments.} 
I am very grateful to V.~Kac and A.~Joseph for useful comments.

\section{Kac-Moody superalgebras}\label{sect1}
The notions of a Kac-Moody superalgebras and its Weyl group were introduced
in~\cite{S}. We recall some definitions below and then prove Lemmas
~\ref{lemrhowrho},\ref{Rrho}. In the sequel,
we will apply these lemma to the case of affine Lie superalgebras;
in this case  the lemmas can be also
verified using the explicit description of root systems.

\subsection{Construction of Kac-Moody superalgebras}
Let $A=(a_{ij})$ be an $n\times n$-matrix over $\mathbb{C}$
 and let $\tau$ be a subset of $I:=\{1,\ldots,n\}$. 
Let $\fg=\fg(A,\tau)=\fn_-\oplus\fh\oplus\fn_+$ 
be the associated  Lie superalgebra constructed as in~\cite{K77},\cite{Kbook}.
Recall that, in order to construct $\fg(A,\tau)$, one considers a
realization of $A$, i.e. a triple $(\fh,\Pi,\Pi^{\vee})$, where $\fh$ 
is a vector space of dimension $n+corank A$, $\Pi\subset\fh^*$
(resp. $\Pi^{\vee}\subset\fh$) is a linearly independent set of vectors
$\{\alpha_i\}_{i\in I}$ (resp. $\{\alpha_i^{\vee}\}_{i\in I}$),
such that $\langle \alpha_i,\alpha_j^{\vee}\rangle=a_{ji}$,
and constructs a Lie superalgebra $\tilde{\fg}(A,\tau)$ 
on generators $e_i,f_i,\fh$, subject to relations:
$$\begin{array}{l}
[\fh,\fh]=0,\ \ [h,e_i]=\langle \alpha_i,h\rangle e_i,\ \
[h,f_i]=-\langle \alpha_i,h\rangle f_i,\ \text{ for } i\in I,h\in\fh,\ \ 
[e_i,f_j]=\delta_{ij}\alpha_i^{\vee},\\
\ p(e_i)=p(f_i)=\ol{1} \text{ if } i\in\tau,\ \
p(e_i)=p(f_i)=\ol{0} \text{ if } i\not\in\tau,\ \ \ p(\fh)=\ol{0}.
\end{array}$$
Then $\fg(A,\tau)=\tilde{\fg}(A,\tau)/J=\fn_-\oplus\fh\oplus\fn_+$, 
where $J$ is the maximal ideal of
$\tilde{\fg}(A,\tau)$, intersecting $\fh$ trivially, and
$\fn_+$ (resp. $\fn_-$) is the subalgebra generated by the images of the
$e_i$'s (resp. $f_i$'s). 
We obtain the triangular decomposition $\fg(A)=\fn_-\oplus\fh\oplus\fn$. 

Let $\Delta$ be the set of roots of $\fg(A)$, i.e. 
$\Delta=\{\alpha\in\fh^*| \alpha\not=0\ \&\ \fg_{\alpha}\not=0\},\ 
\Delta_+=\{\alpha\in\fh^*| \fn_{\alpha}\not=0\},\ 
\Delta_-=\{\alpha\in\fh^*| \fn_{-,\alpha}\not=0\}$.
One has $\Delta=\Delta_+\coprod \Delta_-$, $\Delta_-=-\Delta_+$.

We say that a simple root $\alpha_i$ is {\em even} (resp., {\em odd}) 
if $i\not\in\tau$ (resp., $i\in\tau$) and 
that $\alpha_i$ is {\em isotropic} if $a_{ii}=0$. 
One readily sees that if $i\in\tau$ (i.e., $e_i,f_i$ are odd), 
then $[e_i,e_i],[f_i,f_i]\in J$ iff $a_{ii}=0$. Therefore for
a simple root $\alpha$ one has
$2\alpha\in\Delta$ iff $\alpha$ is a non-isotropic and odd.

Note that, multiplying the $i$-th row of the matrix $A$ by a non-zero number
corresponds to multiplying $e_i$ and $\alpha^{\vee}_i$ by this number,
thus giving an isomorphic Lie superalgebra. Hence we may assume 
from now on that $a_{ii}=2$ or $0$ for all $i\in I$.

\subsubsection{}\label{condA}
We consider the case when the Cartan matrix $A=(a_{ij})$ is such that

(1) $a_{ii}\in\{0,2\}$ for all $i\in I$ and $a_{ij}=0$ forces $a_{ji}=0$;

(2) if $i\not\in\tau$, then  
$a_{ii}=2$ and $a_{ij}\in\mathbb{Z}_{\leq 0}$ for $j\not=i$;

(3) if $i\in \tau$ and $a_{ii}=2$, then
$a_{ij}\in 2\mathbb{Z}_{\leq 0}$ for $j\not=i$.

In this case $\ad e_i,\ad f_i$ act locally nilpotently for each $i\in I$.

\subsection{Weyl group}\label{odd}
Recall a notion of odd reflections, see~\cite{S}. Let $\Pi$ be a set of simple
roots and $\Delta_+$ be the corresponding set of positive roots. 
Fix a simple regular isotropic root $\beta\in\Pi$ and set
$s_{\beta}(\Pi):=\{s_{\beta}(\alpha)|\ \alpha\in \Pi\}$, where
$$ \text{ for }\alpha\in\Pi\ \ \ \ \begin{array}{lll}
s_{\beta}(\alpha)=-\alpha, & s_{\beta}(\alpha^{\vee})=\alpha^{\vee} &
\text{ if } \alpha=\beta,\\
s_{\beta}(\alpha)=\alpha, & s_{\beta}(\alpha^{\vee})=\alpha^{\vee} &
\text{ if } a_{\alpha \beta}=0,\alpha\not=\beta,\\
s_{\beta}(\alpha)=\alpha+\beta, & s_{\beta}(\alpha^{\vee})=
a_{\alpha \beta}\beta^{\vee}+a_{\beta \alpha}\alpha^{\vee} & 
a_{\alpha\beta}\not=0, a_{\alpha\alpha}+2a_{\alpha\beta}=0,\\
s_{\beta}(\alpha)=\alpha+\beta, & s_{\beta}(\alpha^{\vee})=
2\frac{a_{\alpha \beta}\beta^{\vee}+a_{\beta \alpha}\alpha^{\vee}}
{a_{\beta\alpha}(a_{\alpha\alpha}+2a_{\alpha\beta})}\ & 
a_{\alpha\beta}, a_{\alpha\alpha}+2a_{\alpha\beta}\not=0.
\end{array}
$$
One has $\langle \alpha,\alpha^{\vee}\rangle\in\{0,2\}$ for
each $\alpha\in s_{\beta}(\Pi)$.

By~\cite{S}, Sect. 3, $s_{\beta}(\Pi)$ is a set of simple roots 
for $\Delta$ and the corresponding set of positive roots is
$s_{\beta}(\Delta_+):=\Delta_+\setminus\{\beta\}\cup\{-\beta\}$.
The Cartan matrix corresponding to $s_{\beta}(\Pi)$ is
$(\langle\alpha^{\vee},\alpha'\rangle)_{\alpha,\alpha'\in s_{\beta}(\Pi)}$.

\subsubsection{}\label{condA1}
We assume that $\fg(A)$ is such that
{\em for any chain of odd reflections, the corresponding 
Cartan matrix satisfies the conditions (1)-(3) of~\ref{condA}}.
By~\cite{S} Section 6, the finite-dimensional Kac-Moody superalgebras
and their affinizations satisfy this assumption; other examples and 
classification are given in~\cite{S},\cite{HS}.

\subsubsection{}\label{Theta}
Let $\Theta$ be the collection of all possible
sets of simple roots obtained from $\Pi$
by finite sequences of odd reflections. 

\subsubsection{}
\begin{defn}{} 
An even root $\alpha\in\Delta$ is called {\em principal} if 
$\alpha\in\Pi'$ or $\frac{1}{2}\alpha\in\Pi'$ for some 
$\Pi'\in \Theta$. 
\end{defn}

\subsubsection{}\label{pral}
For each principal root $\alpha$ we fix
$\alpha^{\vee}$ as follows: we choose $\Pi'\in \Theta$ such that
$\alpha\in\Pi'$ or $\frac{1}{2}\alpha\in\Pi'$; in first case,
we take $\alpha^{\vee}\in(\Pi')^{\vee}$ and in the second case
we take $\alpha^{\vee}:=(\frac{1}{2}\alpha)^{\vee}/2$, where
$(\frac{1}{2}\alpha)^{\vee}\in(\Pi')^{\vee}$.
Thanks to the assumption~\ref{condA1}, $\langle\beta,\alpha^{\vee}\rangle
\in\mathbb{Z}$ for each $\beta\in\Pi'$. Thus
for each principal root $\alpha$ one has $\langle\Delta,\alpha^{\vee}\rangle
\subset\mathbb{Z}$.

The matrix $A$  is called {\em symmetrizable}
if for some invertible diagonal matrix $D$ the product
 $DA$ is a symmetric matrix. If $A$ is symmetrizable,
then $\fg(A)$ admits a non-degenerate invariant bilinear
form and the restriction of this form induces a non-degenerate 
bilinear form $(-,-)$ on $\fh^*$. In this case, for each principal root 
$\alpha$ the coroot $\alpha^{\vee}$ is given by  the formula
$\langle\mu,\alpha^{\vee}\rangle=\frac{2(\mu,\alpha)}{(\alpha,\alpha)}$
for any $\mu\in\fh^*$.

\subsubsection{}
Take $\Pi'\in \Theta$. Recall that if $\alpha\in\Pi'$ is odd and is
such that $\langle\alpha,\alpha^{\vee}\rangle\not=0$, then 
$2\alpha$ is a root. Thus $\alpha\in\Pi'$ is principal
iff $\langle\alpha,\alpha^{\vee}\rangle\not=0$.

Since the odd reflections do not change the set of even
positive roots, all principal roots are positive.

For a principal root $\alpha$ let $s_{\alpha}\in GL(\fh^*)$
be the reflection $\mu\mapsto \mu-\langle \mu,\alpha^{\vee}\rangle\alpha$.
If $\alpha\in\Pi$, then $s_{\alpha}(\Delta_+(\Pi)\setminus\{\alpha\})=
\Delta_+(\Pi)\setminus\{\alpha\}$. If $\alpha/2\in\Pi$, 
then $s_{\alpha}(\Delta_+(\Pi)\setminus\{\alpha,\alpha/2\})=
\Delta_+(\Pi)\setminus\{\alpha,\alpha/2\}$.

\subsubsection{}
\begin{defn}{} 
The {\em Weyl group} $W$ is the subgroup of $GL(\fh^*)$ 
generated by the reflections $s_{\alpha}$
with respect to the principal roots.
Clearly, $\det s_{\alpha}=-1$ so $\det w=\pm 1$ for each $w\in W$.
Denote by $\sgn: W\to\{\pm 1\}$ the group homomorphism 
$\sgn(w):=\det w$. 
\end{defn}

One has $W\Delta=\Delta$.

\subsubsection{}
\begin{rem}{}
Let $\fg$ be a finite-dimensional  Kac-Moody superalgebra. By~\cite{S},
$\fg$ satisfies the assumption~\ref{condA1}.
Let $\Delta$ be the root system of $\fg$. In this case
$\fg_0$ is a reductive Lie algebra so $\Delta_{\ol{0}}$
is a root system of finite type. The set of principal roots in $\Delta$ 
is a set of simple roots in $\Delta_{\ol{0}}$
(corresponding to the set of positive roots $\Delta_{\ol{0}}\cap\Delta_+$).
In particular, the Weyl group of $\fg$ coincides with the Weyl group of 
$\Delta_{\ol{0}}$.

Consider the case when $\fg\not=\fgl(n,n)$. The affinization 
$\fhg$ of $\fg$ is a  Kac-Moody superalgebra, 
satisfying the assumption~\ref{condA1} (see~\cite{S}). 
Let $\hat{\Delta}$ be the 
root system of $\fhg$.  In this case $\hat{\Delta}_{\ol{0}}$
is a disjoint union of affine root systems (which are the affinizations
of irreducible components of $\Delta_{\ol{0}}$) and the set of principal roots 
in $\hat{\Delta}$ is a set of simple roots in $\hat{\Delta}_{\ol{0}}$
(corresponding to the set of positive roots 
$\hat{\Delta}_{\ol{0}}\cap\hat{\Delta}_+$).
In particular, the Weyl group of $\fhg$ coincides with the Weyl group of 
$\hat{\Delta}_{\ol{0}}$, so it is the direct product of affine Weyl groups.
\end{rem}

\subsubsection{}
In the sequel we will use the following lemma.

\begin{lem}{Rw}
For any $w\in W$ the set $R(w):=\Delta_+\cap w^{-1}\Delta_-$
is finite.
\end{lem}
\begin{proof}
Let  $\alpha$ be a principal root, i.e.
$\alpha\in\Pi'$ or $\frac{1}{2}\alpha\in\Pi'$
for some $\Pi'\in\Theta$; let
$\Delta'_+$ be the corresponding set of the positive roots.
By above, $\Delta'_+\cap s_{\alpha}\Delta_-'\subset
\{\alpha,\frac{1}{2}\alpha\}$.
Therefore $\Delta_+\cap s_{\alpha}(\Delta_-)\subset \{\alpha,
\frac{1}{2}\alpha\}\cup
\bigl(\Delta_+\setminus\Delta'_+\bigr)$. Since $\Pi'\in\Theta$, the set
$\Delta_+\setminus\Delta'_+$ is  finite. Hence 
$R(s_{\alpha})$ is  finite as well.

Now let $w=s_{\alpha}y$, where $y\in W$ is such that $R(y)$ is finite.
One has 
$$R(w)\subset R(y)\cup 
\{\gamma\in\Delta_+|\ y\gamma\in\Delta_+\cap s_{\alpha}(\Delta_-)\}
\subset R(y)\cup y^{-1} R(s_{\alpha}),$$
so $R(w)$ is finite. The claim follows.
\end{proof}

\subsection{}\label{Q+}
Set
$$Q^+=\sum_{\alpha\in \Pi}\mathbb{Z}_{\geq 0}\alpha,\ \ \ 
P:=\{\lambda\in\fh^*|\ \langle\lambda,\alpha^{\vee}\rangle\in\mathbb{Z}
\text{ for all $\alpha\in\Pi$ s.t. } 
\langle\alpha,\alpha^{\vee}\rangle\not=0\}.$$
Clearly, $P$ is an additive subgroup of $\fh^*$.
The conditions on the Cartan matrix in~\ref{condA} 
ensure that $\Delta\subset P$; in particular, $P-Q^+\subset P$.
Introduce the standard partial order on $P$ by
$\mu\leq\nu$ if $(\nu-\mu)\in Q^+$.
Introduce the height function $\htt:Q^+\to\mathbb{Z}_{\geq 0}$ by 
$\ \htt (\sum_{\alpha\in {\Pi}}m_{\alpha}\alpha):=
\sum_{\alpha\in{\Pi}}m_{\alpha}$.

\subsubsection{}\label{rho}
Choose $\rho$ such that
$\langle\rho,\alpha^{\vee}\rangle=\frac{1}{2}
\langle\alpha,\alpha^{\vee}\rangle$
for each $\alpha\in\Pi$. For each $\Pi'\in\Theta$ 
set $\rho_{\Pi'}:=\rho+\sum_{\beta\in \Delta_+(\Pi)\setminus\Delta_+(\Pi')}
\beta$. One readily sees  that $\langle\rho_{\Pi'},\alpha^{\vee}\rangle=
\frac{1}{2}\langle\alpha,\alpha^{\vee}\rangle$
for each $\alpha\in\Pi'$. The assumption~\ref{condA1} ensures
that  $\langle\beta,\alpha^{\vee}\rangle\in\mathbb{Z}$ for each 
$\alpha\in\Pi'$ such that $\langle\alpha,\alpha^{\vee}\rangle\not=0$. 
We conclude that $\langle\rho,\alpha^{\vee}\rangle\in\mathbb{Z}$ for each
$\alpha\in\Pi'$ such that $\langle\alpha,\alpha^{\vee}\rangle\not=0$ 
and each $\Pi'\in\Theta$. In particular, 
$W\rho\subset (\rho+\sum_{\alpha\in {\Delta}}\mathbb{Z}\alpha)$.

\subsubsection{}
\begin{lem}{lemrhowrho}
Let $\Pi_+$ be the set of principal roots satisfying
$\langle\rho,\alpha^{\vee}\rangle\geq 0$ and
let $W_+$ be the subgroup of $W$ generated by the reflections
$\{s_{\alpha},\alpha\in\Pi_+\}$. 

(i) One has $\rho-w\rho\in Q^+$ for any $w\in W_+$.

(ii) If $w=s_{\alpha_{i_1}}\ldots s_{\alpha_{i_r}}$ is a reduced decomposition
of $w\in {W}_+$, then
$$\htt(\rho-w\rho)\geq |\{j: \langle\rho,\alpha_{i_j}\rangle\not=0\}|.$$

(iii) The stabilizer of $\rho$ in $W_+$ is generated by the reflections
$\{s_{\alpha}|\ \alpha\in\Pi_+\ \&\ \langle\rho,\alpha\rangle=0\}$. 
\end{lem}
\begin{proof}
By~\cite{S}, Cor. 4.10, $W$ is the Weyl group of
a Kac-Moody algebra, whose set of simple roots coincides with
 the set of principal roots in $\Delta$. 
Therefore $W_+$ is the Weyl group of
a Kac-Moody algebra, whose set of simple roots coincides with $\Pi_+$.
For $w\in W_+$, let $l(w)$ be the length of $w$. 
Write $w=w's_{\alpha}$, where $l(w)>l(w')$ and 
$\alpha\in \Pi_+$. By~\cite{Jbook}, A.1, the inequality $l(w)>l(w')$ 
implies that $w'\alpha$ is a non-negative linear combination
of elements of $\Pi_+$, so $w'\alpha\in\Delta_+$.
One has
$$\rho-w\rho=\rho-w'\rho+\langle\rho,\alpha\rangle w'\alpha.$$

By~\ref{rho}, $\langle\rho,\alpha\rangle\in\mathbb{Z}$. Since
$\alpha\in\Pi_+$, one has $\langle\rho,\alpha\rangle\in\mathbb{Z}_{\geq 0}$ so
$\langle\rho,\alpha\rangle w'\alpha\in Q^+$ and 
$\langle\rho,\alpha\rangle w'\alpha=0$
 iff $\langle\rho,\alpha\rangle=0$. The assertions (i), (ii)
follow by induction on the length of $w$; (iii) follows from (ii).
\end{proof}

\subsection{The algebra $\cR$}\label{cRPi}
Call a {\em $Q^+$-cone} a set of the form $(\lambda-Q^+)$, where 
$\lambda\in\fh^*$.

For a formal sum of the form $Y:=\sum_{\nu\in P} b_{\nu} e^{\nu},\ 
b_{\nu}\in\mathbb{Q}$ define the {\em support} of $Y$ by
$\supp(Y):=\{\nu|\ b_{\nu}\not=0\}$.
Let $\cR$ be a vector space over $\mathbb{Q}$,
spanned by the sums of the form
$\sum_{\nu\in Q^+} b_{\nu} e^{\lambda-\nu}$, where $\lambda\in P,\ 
b_{\nu}\in\mathbb{Q}$. In other words, $\cR$ consists of
the formal sums $Y=\sum_{\nu\in P} b_{\nu}e^{\nu}$ with the support 
lying in a finite union of $Q^+$-cones. Note that for any non-zero $Y\in\cR$
the support of $Y$ has a maximal element (with respect to the order 
introduced in~\ref{Q+}).

Clearly, $\cR$ has a structure of commutative algebra over 
$\mathbb{Q}$. One has 
If $Y\in \cR$ is such that $YY'=1$ for some $Y'\in\cR$,
we write $Y^{-1}:=Y'$.

\subsubsection{Action of the Weyl group}
For $w\in W$ set $w(\sum_{\nu\in P} b_{\nu}e^{\nu}):=
\sum_{\nu\in P} b_{\nu}e^{w\nu}$. One has $wY\in\cR$ iff
$w(\supp Y)$ is a subset of a finite union of $Q^+$-cones.

Let $W'$ be a subgroup of $W$. Let $\cR_{W'}:=\{Y\in\cR|\ wY\in \cR \text{ for
each }w\in W'\}$. Clearly, $\cR_{W'}$ is a subalgebra of $\cR$.

\subsubsection{Infinite products}\label{infprod}
An infinite product of the form $Y=\prod_{\alpha\in X}
(1+a_{\alpha}e^{-\alpha})^{r(\alpha)}$, where $a_{\alpha}\in \mathbb{Q},\
\ r(\alpha)\in\mathbb{Z}_{\geq 0}$ and $X\subset \Delta$ is such that 
the set $X\setminus\Delta_+$ is finite, can be naturally viewed 
as an element of $\cR$; clearly, this element does not depend
on the order of factors. Let $\cY$ be the set of such infinite products.
For any $w\in W$ the infinite product
$$wY:=\prod_{\alpha\in X}(1+a_{\alpha}e^{-w\alpha})^{r(\alpha)},$$
is again an infinite product of the above form, since
 the set $w\Delta_+\setminus \Delta_+=-(w\Delta_-\cap\Delta_+)$ 
is finite by~\Lem{Rw}.
Hence $\cY$ is a $W$-invariant multiplicative subset of $\cR_W$.

It is easy to see that the elements of $\cY$ are invertible in $\cR$: using 
the geometric series we can expand  $Y^{-1}$ 
(for example, for $\alpha\in\Delta_+$ one has $(1-e^{\alpha})^{-1}=
-e^{-\alpha}(1-e^{-\alpha})^{-1}=-\sum_{i=1}^{\infty} e^{-i\alpha}$).

\subsubsection{The subalgebra $\cR'$}
Denote by $\cR'$ the localization of $\cR_W$ by $\cY$. By above,
$\cR'$ is a subalgebra of $\cR$. Observe that $\cR'\not\subset \cR_W$:
for example, $(1-e^{-\alpha})\in\cR'$, but 
$(1-e^{-\alpha})^{-1}=\sum_{j=0}^{\infty} e^{-j\alpha}\not\in \cR_{W}$.
We extend the action of $W$
from $\cR_W$ to $\cR'$ by setting $w(Y^{-1}Y'):=(wY)^{-1}(wY')$.

An infinite product of the form $Y=\prod_{\alpha\in X}
(1+a_{\alpha}e^{-\alpha})^{r(\alpha)}$, where $a_{\alpha}, X$ are as above
and $r(\alpha)\in\mathbb{Z}$ lies in $\cR'$ and
$wY=\prod_{\alpha\in X} (1+a_{\alpha}e^{-w\alpha})^{r(\alpha)}$.
One has
$$\supp(Y)\subset\lambda'-Q^+,\ \text{ where }
\lambda':=-\sum_{\alpha\in X\setminus\Delta_+: a_{\alpha}\not=0} 
r_{\alpha}\alpha.$$

\subsubsection{}\label{compex}
Let $W'$ be a subgroup of $W$. 
For $Y\in\cR'$ we say that {\em $Y$ is $W'$-invariant
(resp., $W'$-skew-invariant)} if $wY=Y$
(resp., $wY=\sgn(w)Y$) for each $w\in W'$.

Let $Y=\sum a_{\mu} e^{\mu}\in\cR_{W'}$ be $W'$-skew-invariant. 
Then $a_{w\mu}=(-1)^{\sgn(w)}a_{\mu}$
for each $\mu$ and $w\in W'$. In particular,  
$W'\supp(Y)=\supp(Y)$, and, moreover, for each  
$\mu\in\supp(Y)$ one has $\Stab_{W'}\mu\subset\{w\in W'|\ \sgn(w)=1\}$.
The condition $Y\in \cR_{W'}$ is essential: for example, for
$W'=\{\id,s_{\alpha}\}$,
the expression $Y:=e^{\alpha}-e^{-\alpha}$ is $W'$-skew-invariant, so
$Y^{-1}=e^{-\alpha}(1-e^{-2\alpha})^{-1}$ is also $W'$-skew-invariant,
but $\supp(Y^{-1})=-\alpha,-3\alpha,\ldots$ is not 
$s_{\alpha}$-invariant.

Take $Y=\sum a_{\mu} e^{\mu}\in\cR_{W'}$ and set
$\sum_{w\in W'}\sgn(w) wY=:\sum b_{\mu} e^{\mu}$.
One has $b_{\mu}=\sum_{w\in W'}\sgn(w) a_{w\mu}$
so $b_{\mu}=\sgn(w)b_{w\mu}$
for each $w\in W'$. We conclude that 
$$Y\in\cR_{W'}\ \&\ 
\sum_{w\in W'}\!\sgn(w) wY\in\cR\ \Longrightarrow\ \left\{\begin{array}{l}
\sum_{w\in W'}\sgn(w) wY\in \cR_{W'};\\  \sum_{w\in W'}\sgn(w) wY
\text{ is $W'$-skew-invariant};\\ 
\supp (\sum_{w\in W'}\sgn(w) wY)\text{ is  $W'$-stable}.
\end{array}\right.$$

\subsection{}\label{RPi}
For each $\Pi'\in\Theta$ (see~\ref{Theta}) introduce  the following
elements of $\cR$:
$$R(\Pi')_0:=\prod_{\alpha\in\Delta_+(\Pi')\cap\Delta_{\ol{0}}} 
(1-e^{-\alpha}),\ \ \ \ 
R(\Pi')_1:=\prod_{\alpha\in\Delta_+(\Pi')\cap\Delta_{\ol{1}}} 
(1+e^{-\alpha}),\ \ R(\Pi'):=\frac{R(\Pi')_0}{R(\Pi')_1}.$$
We set
$$R_0:=R(\Pi)_0,\ \ R_1:=R(\Pi)_1,\ \ R:=R(\Pi).$$
One readily sees from~\ref{rho} that $R(\Pi')e^{\rho_{\Pi'}}=Re^{\rho}$
for any $\Pi'\in\Theta$. 

\subsubsection{}
\begin{lem}{Rrho}
$Re^{\rho}$ is a $W$-skew-invariant element of $\cR'$.
\end{lem}
\begin{proof}
By~\ref{infprod}, $R_0,R_1\in\cY$ so $Re^{\rho}\in\cR'$.
Let $\alpha$ be a  principal root. If $\alpha\in\Pi'$, then
$s_{\alpha}(\Delta_+'\setminus\{\alpha\})=
\Delta'_+\setminus\{\alpha\}$.
If $\alpha/2\in\Pi'$, then $s_{\alpha}(\Delta_+'\setminus\{\alpha,\alpha/2\})=
\Delta'_+\setminus\{\alpha,\alpha/2\}$. In both cases 
$s_{\alpha}(R(\Pi')e^{\rho_{\Pi'}})=-R(\Pi')e^{\rho_{\Pi'}}$.
By~\ref{RPi}, $R(\Pi')e^{\rho_{\Pi'}}=Re^{\rho}$.
The claim follows.
\end{proof}

\section{Proof of the denominator identity}
We retain notation of Sect.~\ref{intro}.

Fix triangular decomposition of the reductive Lie algebra
$\fg_0$. By~\cite{S}, any two sets of simple roots of $\fg$, 
which are compatible
with a triangular decomposition of the reductive Lie algebra
$\fg_0$,  are connected by a chain of odd reflections.
By~\ref{RPi}, both sides of 
the denominator identity 
$\hat{R}e^{\hat{\rho}}=\sum_{w\in T} w(Re^{\hat{\rho}})$
do not change if we substitute $\Pi$ by $s_{\beta}\Pi$, where $\beta$
is a simple odd root of $\fg$.  
Hence it is enough to prove the denominator identity 
for one choice of $\Pi$; this is done in this section.

\subsection{Another form of denominator identity}
Let us recall the denominator identity for $\fg$ (see~\cite{KW},\cite{G}
for a proof).

Recall that $S\subset \Delta_{\ol{1}}$ is called a 
{\em maximal isotropic set of roots}
if $S$ is a basis of a maximal isotropic subspace in $\fh^*$ with respect to
the form $(-,-)$.
By~\cite{KW}, there exists a maximal isotropic set of roots $S$ and each 
such $S$ is a subset of a set of simple roots (for each $S$ there exists
$\Pi$ such that $S\subset \Pi$).

Fix a maximal isotropic set of roots $S$
and a set of simple roots $\Pi$ such that $S\subset \Pi$.
The denominator identity for $\fg$ takes the following form:
\begin{equation}\label{fincase}
Re^{\rho}=\sum_{w\in W^{\#}}\sgn(w) w\bigl(\frac{e^{\rho}}{\prod_{\beta\in
S}(1+e^{-\beta})} \bigr),
\end{equation}
where $\rho\in\fh^*$ is such that $2(\rho,\alpha)=(\alpha,\alpha)$
for each $\alpha\in\Pi$. Note that
$\hat{\rho}-\rho$ is $W^{\#}$-invariant. Using~(\ref{fincase}) we obtain
$$\begin{array}{ll}
\sum_{y\in T}  y(Re^{\hat{\rho}})&=\sum_{y\in T}  y\bigl(
e^{\hat{\rho}-\rho}Re^{\rho}\bigr)=\sum_{y\in T}  y\bigl(
e^{\hat{\rho}-\rho}\sum_{w\in W^{\#}}
\sgn(w) w\bigl(\frac{e^{\rho}}{\prod_{\beta\in S}(1+e^{-\beta})}
\bigr)\bigr)\\
&=\sum_{y\in T}  y\bigl(\sum_{w\in W^{\#}}
\sgn(w) w\bigl(\frac{e^{\hat{\rho}}}{\prod_{\beta\in S}(1+e^{-\beta})}
\bigr)\bigr)=\sum_{w\in\hat{W}^{\#}} w
\bigl(\frac{e^{\hat{\rho}}}{\prod_{\beta\in S}(1+e^{-\beta})}
\bigr).
\end{array}$$
Hence the denominator identity for $\fhg$ can be rewritten as
\begin{equation}\label{denom}
\hat{R}e^{\hat{\rho}}=\sum_{w\in\hat{W}^{\#}} \sgn(w) w
\bigl(\frac{e^{\hat{\rho}}}{\prod_{\beta\in S}(1+e^{-\beta})}
\bigr).\end{equation}

We set
$$Y:=\sum_{w\in \hat{W}^{\#}}\sgn(w)w(\frac{e^{\hat{\rho}}}{\prod_{\beta\in S}
(1+e^{-\beta})}).$$

\subsection{Notation}\label{condii}
We set
$$\Delta_2=\{\alpha\in\Delta_{\ol{0}}|\ (\alpha,\alpha)<0\}.$$
Then $\Delta_{\ol{0}}=\Delta^{\#}\coprod\Delta_2$, and $\Delta^{\#},\Delta_2$ 
are root systems of semisimple Lie algebras.

Denote by $\delta$ the minimal imaginary root in $\hat{\Delta}$.
Let $\Pi$ be a set of simple roots for $\Delta_+$ and $\theta\in\Delta_+$
be a maximal root.
Recall that $\hat{\Pi}=\Pi\cup\{\delta-\theta\}$ is the set of simple roots
for $\hat{\Delta}_+=\cup_{s=1}^{\infty} \{s\delta+\Delta\}\cup \Delta_+$.

For $\fg\not=B(n,n)$, we fix a set of simple roots 
$\Pi$ for $\Delta$ such that
$$\begin{array}{ll}
(i)&  \Pi\ \text{ contains a maximal isotropic set of roots } S;\\

(ii) & \forall \alpha\in\Pi\ \ (\alpha,\alpha)\geq 0; \\

(iii)  & \theta\in\Delta^{\#},
\end{array}
$$
see~\ref{rootsys} for a choice of $\Pi$. 
For $\fg=B(n,n)$ we choose $\Pi$ as in~\ref{rootsys};
in this case the properties (i) and (ii) hold, but $\theta$ is  isotropic. 
Note that in all cases $(\theta,\theta)\geq 0$.
Combining with (ii), we get $(\hat{\rho},\beta)=(\beta,\beta)/2\geq 0$ 
for all $\beta\in\hat{\Pi}$. Set
$\hat{Q}^+:=\sum_{\alpha\in\Pi} \mathbb{Z}_{\geq 0}\alpha$. 
We obtain
$$(\hat{\rho},\hat{Q}^+)\geq 0,\ \ \ 
\frac{(\hat{\rho},\alpha)}{(\alpha,\alpha)}\geq 0\ \text{ for }
\alpha\in\hat{\Delta}^{\#}.$$

\subsection{Support of $\hat{R}e^{\hat{\rho}}$}
Set
$$U:=\{\mu\in\hat{\rho}-\hat{Q}^+|\ (\mu,\mu)=(\hat{\rho},\hat{\rho})\}.$$
From representation theory we know that the character of the trivial
$\fhg$-module is a linear combination of the characters of 
Verma $\fhg$-modules $M(\lambda)$, where $\lambda\in-\hat{Q}$ and
$(\lambda+\hat{\rho},\lambda+\hat{\rho})=(\hat{\rho},\hat{\rho})$
(since $\fhg$ admits the Casimir element). Therefore
$1=\sum_{\lambda\in U-\hat{\rho}} a_{\lambda}\hat{R}^{-1}e^{\lambda}$
that is
$$\supp (\hat{R}e^{\hat{\rho}})\subset U.$$

\subsection{Support of $Y$}
Expanding the summands of $Y$ we  obtain
$$\supp\bigl(\frac{e^{w\hat{\rho}}}{\prod_{\beta\in S}
(1+e^{-w\beta})}\bigr)\subset \{w\hat{\rho}-\hat{Q}^+\}
\cap\{w\hat{\rho}+\sum_{\beta\in S}\mathbb{Z}w\beta\}.$$
Since $(\hat{\rho},S)=(S,S)=0$ this implies

\begin{equation}\label{expwfrac}
\supp\bigl(\frac{e^{w\hat{\rho}}}{\prod_{\beta\in S}
(1+e^{-w\beta})}\bigr)\subset \{\mu\in w\hat{\rho}-\hat{Q}^+|\
(\mu,\mu)=(\hat{\rho},\hat{\rho})\}.
\end{equation}

\subsubsection{}
\begin{lem}{lem1}
(i) $Y$ is a well-defined element of $\cR$ (see~\ref{cRPi}
for the notation);

(ii) $\supp(Y)\subset U$;

(iii) for $\fg\not=B(n,n)$ the coefficient of $e^{\hat{\rho}}$ in 
$Y$ is equal to $1$.
\end{lem}
\begin{proof}
By~\cite{S}, the set of principal roots
of $\hat{\Delta}_+$ is the set of
simple roots of $\hat{\Delta}_{+,0}$. By~\ref{condii},
$\hat{W}^{\#}$ is a subgroup of the group $W_+$ introduced
in~\Lem{lemrhowrho}. By above,
$\supp\bigl(\frac{e^{w\hat{\rho}}}{\prod_{\beta\in S}
(1+e^{-w\beta})}\bigr)\subset w\hat{\rho}-\hat{Q}^+$.
In the light of~\Lem{lemrhowrho} for (i) it is enough to show that 
$H_r:=\{w\in \hat{W}^{\#}|\ \htt(\hat{\rho}-w\hat{\rho})\leq r\}$
is finite for each $r$.

Let $\Sigma$ be the set of simple roots of $\hat{\Delta}^{\#}_+$.
Set $\Sigma_0:=\{\alpha\in \Sigma|\  (\hat{\rho},\alpha)=0\}$.
 By~\Lem{lemrhowrho}, $H_0=\Stab_{\hat{W}^{\#}}\hat{\rho}$ 
is the subgroup of $\hat{W}^{\#}$ 
generated by the reflections $\{s_{\alpha}:\alpha\in \Sigma_0\}$
and any $w\in H_r$ is of the form 
$w_1s_{\beta_1}w_2s_{\beta_2}w_3\ldots s_{\beta_r}w_{r+1}$, 
where $w_j\in H_0$ and $\beta_j\in \Sigma\setminus \Sigma_0$.
This means that the finiteness of $H_0$ implies the finiteness of
$H_r$ for $r\geq 0$ and that $H_0$ is the Weyl group of
the Dynkin diagram corresponding to $\Sigma_0$. Hence
for (i) it is enough to verify that the Dynkin diagram of $\Sigma_0$
is of finite type. This can be shown as follows.
Observe that $\Sigma$ is an indecomposable Dynkin diagram of 
affine type. Since $\Sigma_0\subset \Sigma$, it is enough to verify that 
$\Sigma_0\not=\Sigma$.
Since $(\hat{\rho},\delta)=h^{\vee}\not=0$, 
there exists  $\beta\in\hat{\Pi}$ such that
$(\hat{\rho},\beta)\not=0$ that is $(\beta,\beta)\not=0$.
By~\ref{condii}, $(\alpha,\alpha)\geq 0$ for all $\alpha\in\hat{\Pi}$,
so $(\beta,\beta)>0$. Hence  $\beta$ or $2\beta$ belongs
to $\hat{\Delta}^{\#}$. Therefore $(\hat{\rho},\hat{\Delta}^{\#})\not=0$
so $\Sigma_0\not=\Sigma$. This establishes (i).

Combining~(\ref{expwfrac}) and~\Lem{lemrhowrho} (i), 
we obtain $\supp(Y)\subset U$, thus (ii).

Let us show that the coefficient of $e^{\hat{\rho}}$ in 
$Y$ is $1$ for $\fg\not=B(n,n)$. Indeed, by above, 
$\hat{\rho}\in \supp\bigl(\frac{e^{w\hat{\rho}}}{\prod_{\beta\in S}
(1+e^{-w\beta})}\bigr)$ forces $w\in H_0$. 
By~\ref{condii}, for $\fg\not=B(n,n)$
one has $\theta\in\Delta^{\#}$ so $\alpha_0\in \Sigma\setminus \Sigma_0$
and thus $H_0\subset W^{\#}$.
Therefore the coefficient of $e^{\hat{\rho}}$ in $Y$
 is equal to the coefficient of $e^{\hat{\rho}}$
in the expression
$$\sum_{w\in {W}^{\#}}\sgn(w)w(\frac{e^{\hat{\rho}}}{\prod_{\beta\in S}
(1+e^{-\beta})})=e^{\hat{\rho}-\rho}
\sum_{w\in {W}^{\#}}\sgn(w)w(\frac{e^{{\rho}}}{\prod_{\beta\in S}
(1+e^{-\beta})}).$$
Using the denominator identity (\ref{fincase}) we get
$\sum_{w\in {W}^{\#}}\sgn(w)w(\frac{e^{\hat{\rho}}}{\prod_{\beta\in S}
(1+e^{-\beta})})={R}e^{\hat{\rho}}$.
Clearly, the coefficient of $e^{\hat{\rho}}$ in ${R}e^{\hat{\rho}}$
is equal to $1$. This establishes (iii).
\end{proof}

\subsubsection{}
\begin{lem}{lem2}
For $\fg=B(n,n)$ the coefficient of $e^{\hat{\rho}}$ in $Y$
is equal to $1$.
\end{lem}
\begin{proof}
Expanding the expression $w(\frac{e^{\hat{\rho}}}{\prod_{\beta\in S}
(1+e^{-\beta})})$, we see that
the coefficient of $e^{\hat{\rho}}$ in $Y$ is equal to the sum
$\sum_{w\in H}\sgn(w)$, where
$$H:=\{w\in\hat{W}^{\#}|\ w\hat{\rho}=\hat{\rho}\ \&\ 
wS\subset\hat{\Delta}_+\}.$$
Take $w\in H$ and write $w=t_{\mu}y$, where $y\in W^{\#}$ and 
$t_{\mu}\in T$ (see Sect.~\ref{intro} for notation) is
given by 
$$t_{\mu}(\lambda)=\lambda+(\lambda,\delta)\mu-
((\lambda,\mu)+\frac{(\mu,\mu)}{2}(\lambda,\delta))\delta\ \text{ for }
\lambda\in\fhh^*.$$
Retain notation of~\ref{rootsys}. One has $S=\{\delta_i-\vareps_i\}_{i=1}^n$
and $w(\delta_i-\vareps_i)=\delta_i-y\vareps_i+(y^{-1}\mu,\vareps_i)\delta$,
because $(\mu,\Delta_2)=0$, see~\ref{condii}
for notation. The condition $wS\subset\hat{\Delta}_+$
gives $(y^{-1}\mu,\vareps_i)\geq 0$ for $i=1,\ldots,n$. 
On the other hand,
$$w\hat{\rho}=y\hat{\rho}+h^{\vee}\mu-((\hat{\rho},y^{-1}\mu)+
\frac{(\mu,\mu)}{2}h^{\vee})\delta.$$
Since $\mu$ and $\hat{\rho}-y\hat{\rho}$ lie in the $\mathbb{Q}$-span 
of $\Delta^{\#}$,
the condition $w\hat{\rho}=\hat{\rho}$ gives 
$(\hat{\rho},y^{-1}\mu)+\frac{(\mu,\mu)}{2}h^{\vee}=0$.
One has $(\mu,\mu)\geq 0$ since $\mu$ lies in the 
$\mathbb{Q}$-span of $\Delta^{\#}$. Since $h^{\vee}>0$, we get
$$0\geq (\hat{\rho},y^{-1}\mu)=(\rho,y^{-1}\mu)=
(\frac{1}{2}\sum_{i=1}^n\vareps_i,y^{-1}\mu).$$
Using the above inequalities $(y^{-1}\mu,\vareps_i)\geq 0$,
we conclude that $0=(\hat{\rho},y^{-1}\mu)=\frac{(\mu,\mu)}{2}h^{\vee}$ 
that is $\mu=0$. Therefore
$w=y\in W^{\#}$. Now we can obtain the statement
using the argument of the proof of~\Lem{lem1} (iii), or, by
observing that $y\hat{\rho}=\hat{\rho}$ implies that 
$y$ permutes $\{\vareps_i\}_{i=1}^n$ and then $yS\subset\hat{\Delta}_+$
forces $y=\{\id\}$.
\end{proof}

\subsection{}
Assume that the denominator identity does not hold so
$\hat{R}e^{\hat{\rho}}-Y\not=0$. 

The coefficient of $e^{\hat{\rho}}$ in $\hat{R}e^{\hat{\rho}}$
is equal to $1$. From Lemmas~\ref{lem1}, \ref{lem2} we get
\begin{equation}\label{suppR}
\supp\bigl(\hat{R}e^{\hat{\rho}}-Y\bigr)\subset U\setminus\{\hat{\rho}\}.
\end{equation}

Let $\hat{\rho}^{\#}$ be the standard element for the root system
$\hat{\Delta}^{\#}=(\hat{\Delta}^{\#}\cap\hat{\Delta}_+)\coprod 
(\hat{\Delta}^{\#}\cap\hat{\Delta}_-)$. 
Set
$$X:=\hat{R}_1e^{\hat{\rho}^{\#}-\hat{\rho}}(\hat{R}e^{\hat{\rho}}-Y).$$
By the above assumption $X\not=0$.

By~\ref{infprod},  
$\hat{R}_0,\hat{R}_1e^{\hat{\rho}^{\#}-\hat{\rho}}\in \cR_{W}$, where
$W$ is the Weyl group of $\fhg$.
By~\Lem{Rrho}, $\hat{R}e^{\hat{\rho}}, \hat{R}_0e^{\hat{\rho}^{\#}}$ 
are $\hat{W}^{\#}$-skew-invariant elements of
$\cR'$. Therefore for each $w\in\hat{W}^{\#}$ 
one has 
$$\frac{\hat{R}_0 e^{\hat{\rho}^{\#}}}
{\hat{R}_1e^{\hat{\rho}^{\#}-\hat{\rho}}}=
\hat{R}e^{\hat{\rho}}=\sgn(w)w(\hat{R}e^{\hat{\rho}})=\sgn(w)
\frac{w(\hat{R}_0 e^{\hat{\rho}^{\#}})}
{w(\hat{R}_1e^{\hat{\rho}^{\#}-\hat{\rho}})}=
\frac{\hat{R}_0 e^{\hat{\rho}^{\#}}}
{w(\hat{R}_1e^{\hat{\rho}^{\#}-\hat{\rho}})}.$$
Thus $\hat{R}_1e^{\hat{\rho}^{\#}-\hat{\rho}}$ is a
$\hat{W}^{\#}$-invariant element of $\cR_{W}$. Therefore
$$\hat{R}_1e^{\hat{\rho}^{\#}-\hat{\rho}}Y=\sum_{w\in \hat{W}^{\#}} \sgn(w)
w\bigl(\hat{R}_1e^{\hat{\rho}^{\#}}
\prod_{\beta\in S} (1+e^{-\beta})^{-1}\bigr)$$
and so
$$X=\hat{R}_0e^{\hat{\rho^{\#}}}-
\sum_{w\in \hat{W}^{\#}} \sgn(w)wZ,\ \text{ where } 
Z:=e^{\hat{\rho}^{\#}}
\prod_{\beta\in\hat{\Delta}_{\ol{1}+}\setminus S} (1+e^{-\beta}).$$
Since $\hat{\rho}^{\#}-\hat{\rho}\in
\supp(\hat{R}_1e^{\hat{\rho}^{\#}-\hat{\rho}})\subset 
\hat{\rho}^{\#}-\hat{\rho}-\hat{Q}^+$, we obtain from~(\ref{suppR})
$$\max \supp(X)=\hat{\rho^{\#}}-\hat{\rho}+\max \supp 
\bigl(\hat{R}e^{\hat{\rho}}-Y\bigr)\subset 
\hat{\rho^{\#}}-\hat{\rho}+(U\setminus
\{\hat{\rho}\}),$$
that is
\begin{equation}\label{suppX}
\max \supp(X)\subset\{\mu\in\hat{\rho^{\#}}-(\hat{Q}^+\setminus\{0\})|\ 
(\mu-\hat{\rho^{\#}}+\hat{\rho},\mu-\hat{\rho^{\#}}+\hat{\rho})=
(\hat{\rho},\hat{\rho})\}.
\end{equation}

\subsubsection{}\label{regwt}
Recall~\cite{Jbook}, A.1.18, \cite{KT}, that for any $\lambda\in\fhh^*$
the stabilizer $\Stab_{\hat{W}^{\#}} \lambda$ is either trivial or 
contains a reflection $s_{\alpha}$.
Let us call $\lambda\in P$ {\em regular} if $\Stab_{\hat{W}^{\#}} 
\lambda=\{\id\}$.
Say that the orbit $\hat{W}^{\#}\lambda$ is regular if $\lambda$ is regular
(so the orbit consists of regular points).

By~\ref{infprod}, $Z\in \cR_W$. One has 
$\sum_{w\in \hat{W}^{\#}} \sgn(w)wZ=\hat{R}_1e^{\hat{\rho^{\#}}-\hat{\rho}}Y
\in\cR$, because $Y\in\cR$.
In the light of~\ref{compex}, $\sum_{w\in \hat{W}^{\#}} \sgn(w)wZ$ 
belongs to $\cR_{\hat{W}^{\#}}$ and is $\hat{W}^{\#}$-skew-invariant.
By~\Lem{Rrho},  $\hat{R}_{0}
e^{\hat{\rho}^{\#}}\in\cR_{W}$ is $\hat{W}^{\#}$-skew-invariant.
Hence $X$ belongs to $\cR_{\hat{W}^{\#}}$ and is $\hat{W}^{\#}$-skew-invariant.
Using~\ref{compex} we conlclude that
$\supp(X)$ is a union of regular $\hat{W}^{\#}$-orbits.

\subsubsection{}\label{chuki}
Take $\nu\in\max\supp X$. Then $\nu$ is a maximal element
in a regular $\hat{W}^{\#}$-orbit and, by~(\ref{suppX}),
$\nu\in\hat{\rho}^{\#}-(\hat{Q}^+\setminus\{0\})$.

One has $\hat{Q}^+\subset \mathbb{Q}\hat{\Delta}^{\#}+\mathbb{Q} M$, where 
$M=\Delta_2$ (see~\ref{condii} for notation)
if ${\Delta}$ is not of the type $A(m,n), C(n)$, 
and, for the types $A(m,n), C(n)$ one has
 $M=\Delta_2\cup\{\xi\}$, where $\xi\in\mathbb{Q}\Delta$ 
is such that $(\xi,\xi)<0$, 
$(\xi,\Delta_{\ol{0}})=0$. The element $\xi$ is given in~\ref{rootsys}.

Write $\nu=\hat{\rho^{\#}}+\nu_1+\nu_2$, 
where $\nu_1\in \mathbb{Q}\hat{\Delta}^{\#}$
and $\nu_2\in \mathbb{Q} M$.  Since $\hat{W}^{\#}\nu_2=\nu_2$,
the vector $\nu-\nu_2$ is also a maximal element
in a regular $\hat{W}^{\#}$-orbit. For each simple root $\alpha$
of $\hat{\Delta}^{\#}_+$ one has
$$\langle\nu-\nu_2,\alpha^{\vee}\rangle=\langle\nu,\alpha^{\vee}\rangle
\in\langle\hat{\rho}^{\#},\alpha^{\vee}\rangle-
\langle\hat{Q}^+,\alpha^{\vee}\rangle\subset\mathbb{Z},$$
since $\langle\hat{\rho}^{\#},\alpha^{\vee}\rangle=1$ and
$\langle\hat\Delta,\alpha^{\vee}\rangle\subset\mathbb{Z},$
by~\ref{pral}.
In the light of~\Lem{lem3}, $\nu-\nu_2=\hat{\rho}^{\#}+\nu_1
\in \hat{\rho}^{\#}-\mathbb{Q}\delta$
so $\nu_1=-s\delta$ for some $s\in\mathbb{Q}$.

Substituting $\nu_1=-s\delta$ and using~(\ref{suppX})  we get 
\begin{equation}\label{nu12}
(\hat{\rho}-s\delta+\nu_2,\hat{\rho}-s\delta+\nu_2)=(\hat{\rho},\hat{\rho}).
\end{equation}

By~\ref{condii}, 
$(\hat{\rho},-s\delta+\nu_2)\leq 0$, since 
$-s\delta+\nu_2\in -\hat{Q}^+$.
Therefore $(\nu_2,\nu_2)\geq 0$.
Recall that the form $(-,-)$ is negatively definite on $\Delta_2$ and
so is negatively definite on $M$.
Thus $(\nu_2,\nu_2)\geq 0$ gives $\nu_2=0$. Now
the formula~(\ref{nu12})
gives $s=0$ (because $(\hat{\rho},\delta)=h^{\vee}\not=0$).
Hence $\nu=\hat{\rho}^{\#}$, a contradiction.

\section{Appendix}
\subsection{}
The following lemma is used in~\ref{chuki}.
Let $\fg$ be an affine Lie algebra, let $\Pi$ be its
set of simple roots. Let $W$ be the Weyl group of $\fg$.
Define regular $W$-orbits as in~\ref{regwt}. 
\subsubsection{}
\begin{lem}{lem3}
If $\lambda\in \sum_{\alpha\in\Pi}\mathbb{Q}\alpha$ is such that
$\lambda+\rho$ is a maximal element in a regular $W$-orbit
and $\langle \lambda,\alpha^{\vee}\rangle\in\mathbb{Z}$ 
for any $\alpha\in\Pi$, then
$\lambda\in \mathbb{Q}\delta$, where $\delta$ is  
the minimal imaginary root. 
\end{lem}
\begin{proof}
For each $\alpha\in\Pi$ set
$k_{\alpha}:=\langle \lambda,\alpha^{\vee}\rangle$.
Since $\lambda+\rho$ is a maximal element in a regular $W$-orbit
one has 
$s_{\alpha}(\lambda+\rho)=\lambda+\rho-k_{\alpha}\alpha\not\geq\lambda+\rho$,
so $k_{\alpha}\not\in\mathbb{Z}_{\leq 0}$. By the assumption,
$k_{\alpha}$ is an integer, so $k_{\alpha}>0$.

Write $\lambda=\sum_{\alpha\in\Pi} x_{\alpha}\alpha,\ 
x_{\alpha}\in\mathbb{Q}$. Since $\langle \lambda,\alpha^{\vee}\rangle>0$
for each $\alpha\in\Pi$, we have
$Ax\geq 0$, where $A$ is the Cartan matrix
of $\fg$ and $x=(x_{\alpha})_{\alpha\in\Pi}$.
From~\cite{Kbook}, Thm. 4.3 it follows that
$\sum_{\alpha\in\Pi} x_{\alpha}\alpha\in\mathbb{Q}\delta$ as required.
\end{proof}

\subsection{Basic Lie superalgebras}\label{rootsys}
The basic Lie superalgebras with a non-zero Killing form, which are not 
Lie algebras, are
$A(m,n), m\not=n; B(m,n); C(n); D(m,n), m\not=n+1; F_4; G_3$.
Below for each of these root systems we give an example of $\Pi$ satisfying 
the conditions (i), (ii) of~\ref{condii}.
In all cases $\Delta^{\#}$
lies in the lattice spanned by $\{\vareps_i\}_{i=1}^{\max(m,n)}$ and 
$\Delta_2$ 
lies in the lattice spanned by $\{\delta_i\}_{i=1}^{\min(m,n)}$.

Retain notation of~\ref{condii}. 
In all cases except $B(m,m)$ one has
$\theta\in\Delta^{\#}$ (the condition (iii) in~\ref{condii});
for $B(m,m)$ one has $(\theta,\theta)=0$.

In all cases except $A(m,n), C(n)$ one has 
$\mathbb{Q}\Delta=\mathbb{Q}\Delta_{\ol{0}}$; for
$A(m,n), C(n)$ one has 
$\mathbb{Q}\Delta=\mathbb{Q}\Delta_{\ol{0}}+\mathbb{Q}\xi$, where
$\xi$ is described below.

\subsubsection{Case $A(m,n), m\not=n$}
Since $A(m,n)\cong A(n,m)$, we may assume that $m>n$. 
The roots are $\Delta_{\ol{0}}=\{\vareps_i-\vareps_j\}
_{1\leq i\not=j\leq m}\cup\{\delta_i-\delta_j\}
_{1\leq i\not=j\leq n}$, 
$\Delta_{\ol{1}}=\{\pm(\vareps_i-\delta_j)\}_{1\leq i\leq n}
^{1\leq j\leq m}$. Set 
$$\Pi:=\{\vareps_1-\delta_1,\delta_1-\vareps_2,\vareps_2-\delta_2,\ldots,
\delta_n-\vareps_{n+1},\vareps_{n+1}-\vareps_{n+2},\ldots,
\vareps_{m-1}-\vareps_m\}$$
and $S:=\{\vareps_i-\delta_i\}$. One has 
$\theta=\vareps_1-\vareps_m\in\Delta^{\#}$.
We take $\xi=\frac{\sum_{i=1}^m\vareps_i}{m}-\frac{\sum_{j=1}^n\delta_j}{n}$.
One has $(\xi,\Delta_{\ol{0}})=0$.

\subsubsection{Case $B(m,n), m<n$}
The roots are $\Delta_{\ol{0}}=\{\pm\vareps_i\pm\vareps_j;\pm 2\vareps_i\}
_{1\leq i\not=j\leq n}\cup\{\pm\delta_i\pm\delta_j;\pm\delta_i\}
_{1\leq i\not=j\leq m}$, 
$\Delta_{\ol{1}}=\{\pm\vareps_i\pm\delta_j;\pm\vareps_i\}_{1\leq i\leq n}
^{1\leq j\leq m}$. We take
$$\Pi:=\{\vareps_1-\delta_1,\delta_1-\vareps_2,\vareps_2-\delta_2,\ldots,
\vareps_m-\delta_m,\delta_m-\vareps_{m+1},
\vareps_{m+1}-\vareps_{m+2},\ldots,
\vareps_{n-1}-\vareps_n,\vareps_n\}$$
and $S:=\{\vareps_i-\delta_i\}_{i=1}^m$. 
One has $\theta=2\vareps_1\in\Delta^{\#}$.

\subsubsection{Case $B(n,n)$}
The roots are as above. We take
$$\Pi:=\{\delta_1-\vareps_1,\vareps_1-\delta_2,\ldots,
\delta_n-\vareps_{n},\vareps_n\}$$
and $S:=\{\delta_i-\vareps_i\}$.  One has
$\theta=\delta_1+\vareps_1$.

\subsubsection{Case $B(m,n), m\geq n+1$}
The roots are $\Delta_{\ol{0}}=\{\pm\vareps_i\pm\vareps_j;\pm \vareps_i\}
_{1\leq i\not=j\leq m}\cup\{\pm\delta_i\pm\delta_j;\pm 2\delta_i\}
_{1\leq i\not=j\leq n}$, 
$\Delta_{\ol{1}}=\{\pm\vareps_i\pm\delta_j;\pm\vareps_i\}_{1\leq i\leq n}
^{1\leq j\leq m}$. We take for $m=n+2$ 
$$\Pi:=\{\vareps_1-\vareps_2,\vareps_2-\delta_1,\delta_1-\vareps_3,
\ldots,\delta_n-\vareps_{n+2},\vareps_{n+2}\},$$
and for $m>n+2$ 
$$\Pi:=\{\vareps_1-\vareps_2,\vareps_2-\delta_1,\delta_1-\vareps_3,
\ldots,\delta_n-\vareps_{n+2},\vareps_{n+2}-\vareps_{n+3},\ldots,
\vareps_{m-1}-\vareps_{m},\vareps_{m}\}.$$
Then $S:=\{\vareps_i-\delta_i\}_{i=1}^n$ lies in $\Pi$. 
One has $\theta=\vareps_1+\vareps_2\in \Delta^{\#}$.

\subsubsection{Case $C(m)$}
The roots are $\Delta_{\ol{0}}=\{\pm\vareps_i\pm\vareps_j;\pm 2\vareps_i\}
_{1\leq i\not=j\leq m}$, 
$\Delta_{\ol{1}}=\{\pm\vareps_i\pm\delta_1\}_{1\leq j\leq m}$
Set
$$\Pi:=\{\vareps_1-\vareps_2,\vareps_2-\vareps_3,\ldots,
\vareps_{m-1}-\vareps_m,\vareps_m-\delta_1,\vareps_m+\delta_1\}.$$
One has $\theta=2\vareps_1\in \Delta^{\#}$.
Observe that $\Delta_2=\emptyset$. We take $\xi:=\delta_1$.

\subsubsection{Case $D(m,n), n\geq m$}
The roots are $\Delta_{\ol{0}}=\{\pm\vareps_i\pm\vareps_j;\pm 2\vareps_i\}
_{1\leq i\not=j\leq n}\cup\{\pm\delta_i\pm\delta_j\}
_{1\leq i\not=j\leq m}$, 
$\Delta_{\ol{1}}=\{\pm\vareps_i\pm\delta_j\}_{1\leq i\leq n}
^{1\leq j\leq m}$. We take
$$\Pi:=\{\vareps_1-\delta_1,\delta_1-\vareps_2,\vareps_2-\delta_2,\ldots,
\delta_m-\vareps_m,\vareps_m-\vareps_{m+1},\ldots,
\vareps_{n-1}-\vareps_n,2\vareps_n\}$$
and $S:=\{\vareps_i-\delta_i\}_{i=1}^m$.
One has $\theta=2\vareps_1\in \Delta^{\#}$.

\subsubsection{Case $D(m,n), m>n+1$}
The roots are $\Delta_{\ol{0}}=\{\pm\vareps_i\pm\vareps_j\}
_{1\leq i\not=j\leq n}\cup\{\pm\delta_i\pm\delta_j;\pm 2\delta_i\}
_{1\leq i\not=j\leq m}$, 
$\Delta_{\ol{1}}=\{\pm\vareps_i\pm\delta_j\}_{1\leq i\leq n}
^{1\leq j\leq m}$. 
For $m=n+2$ set
$$\Pi:=\{\vareps_1-\vareps_2,\vareps_2-\delta_1,\delta_1-\vareps_3,
\ldots,\delta_n-\vareps_{n+2},\delta_n+\vareps_{n+2}\},$$
for $m>n+2$ set
$$\Pi:=\{\vareps_1-\vareps_2,\vareps_2-\delta_1,\delta_1-\vareps_3,
\ldots,\delta_n-\vareps_{n+2},\vareps_{n+2}-\vareps_{n+3},\ldots,
\vareps_{m-1}-\vareps_{m},\vareps_{m-1}+\vareps_{m}\}.$$
Then $\Pi$ contains $S:=\{\vareps_i-\delta_i\}_{i=1}^n$. 
One has $\theta=\vareps_1+\vareps_2\in \Delta^{\#}$.

\subsubsection{Case $F(4)$}
We choose 
$$\Pi:=\{(\vareps_1+\vareps_2+\vareps_3+\delta_1)/2;
(-\vareps_1+\vareps_2+\vareps_3-\delta_1)/2;
(-\vareps_1-\vareps_2-\vareps_3+\delta_1)/2;
\vareps_1-\vareps_2\}.$$
In this case $S$ can be any odd simple root;
one has 
$\theta=\vareps_3-\vareps_2\in \Delta^{\#}$.

\subsubsection{Case $G(3)$}
For $G(3)$ the roots are expressed in terms of linear functions
$\vareps_1,\vareps_2,\vareps_3$, corresponding to $G_2$,
$\vareps_1+\vareps_2+\vareps_3=0$, and $\delta_1$, corresponding to $A_1$;
we choose 
$\Pi=\{\delta_1-\vareps_2,\vareps_3-\delta_1, -\vareps_3-\vareps_1\}$.
In this case $S$ can be any odd simple root;
one has  $\theta=\vareps_3-\vareps_1\in \Delta^{\#}$.


\end{document}